\documentclass[microtype]{gtpart}
\usepackage{tikz}
\usetikzlibrary{matrix,arrows,decorations.pathmorphing}
\usepackage{graphicx}
\usepackage[frame,ps,matrix,arrow,curve,rotate,all,2cell,tips]{xy}
\usepackage{enumerate}

\newtheorem{theorem}[subsubsection]{Theorem}
\newtheorem{lemma}[subsubsection]{Lemma}

\newtheorem{proposition}[subsubsection]{Proposition}
\newtheorem{corollary}[subsubsection]{Corollary}

\newtheorem{prop}[subsubsection]{Proposition}

\theoremstyle{definition}
\newtheorem{definition}[subsubsection]{Definition}
\newtheorem{example}[subsubsection]{Example}
\newtheorem{remark}[subsubsection]{Remark}

\newtheorem{convention}[subsubsection]{Convention}

\numberwithin{equation}{subsubsection}

\usepackage{tikz}
\usetikzlibrary{arrows,decorations.pathmorphing}
\usetikzlibrary{backgrounds,positioning}
\usetikzlibrary{fit,petri,shapes.misc}

\tikzset{auto}

\tikzset{empty/.style={circle,inner sep=0pt,minimum size=6mm}}
\tikzset{emptyvt/.style={circle,inner sep=0pt,minimum size=0mm}}

\tikzset{plain/.style={circle,draw,very thick,
inner sep=0pt,minimum size=6mm}}

\tikzset{fatplain/.style={rounded rectangle,draw,very thick,minimum size=6mm}}

\tikzset{bigplain/.style={rounded rectangle,draw,very thick,minimum size=.8cm}}

\tikzset{yellowvt/.style={circle,draw,fill=yellow,very thick,inner sep=0pt,minimum size=6mm}}

\tikzset{bluevt/.style={circle,draw,fill=blue!20,very thick,inner sep=0pt,minimum size=6mm}}

\tikzset{greenvt/.style={circle,draw,fill=green!30,very thick,inner sep=0pt,minimum size=6mm}}

\tikzset{redvt/.style={circle,draw,fill=red!30,very thick,inner sep=0pt,minimum size=6mm}}

\tikzset{arrow/.style={->,thick}}
\tikzset{dashedarrow/.style={->,dashed,thick}}
\tikzset{dottedarrow/.style={->,dotted,thick}}
\tikzset{mapto/.style={|->,thick}}

\tikzset{implies/.style={thick,double,double equal sign distance,-implies}}

\tikzset{line/.style={thick}}
\tikzset{dottedline/.style={dotted,thick}}
\tikzset{dashedline/.style={dashed,thick}}

\tikzset{inputleg/.style={<-,thick}}
\tikzset{outputleg/.style={->,thick}}
\tikzset{dottedinput/.style={<-,dotted,thick}}

\newcommand{\adjoint}{
\nicearrow\xymatrix{ \ar@<2pt>[r] & \ar@<2pt>[l]}}
\renewcommand{\hookrightarrow}{\nicexy{\ar@{^{(}->}[r] &}}
\newcommand{\nicearrow}{\SelectTips{cm}{10}}
\newcommand{\nicexy}{\nicearrow\xymatrix@C+10pt@R+10pt}

\newcommand{\pushoutsymbol}{\rotatebox{315}{$\Longrightarrow$}}
\newcommand{\pushout}{\ar@{}[dr]|(0.75){\pushoutsymbol}}

\newcommand{\boxprod}{\mathbin\square}
\newcommand{\comp}{\circ}

\renewcommand{\to}{\longrightarrow}

\newcommand{\dotover}[1]{\underset{#1}{\centerdot}}

\newcommand{\frakC}{\mathfrak{C}}
\newcommand{\fC}{\mathfrak{C}}

\newcommand{\sds}{\mathsf{ds}}

\newcommand{\sI}{\mathsf{I}}
\newcommand{\sJ}{\mathsf{J}}

\newcommand{\tensorunit}{\mathbb{I}}

\newcommand{\sn}{\mathsf{n}}

\newcommand{\sO}{\mathsf{O}}

\newcommand{\sP}{\mathsf{P}}

\newcommand{\ua}{\underline{a}}
\newcommand{\ub}{\underline{b}}
\newcommand{\uc}{\underline{c}}

\newcommand{\smallop}{{\scalebox{.5}{$\text{op}$}}}

\newcommand{\cald}{\mathcal{D}}
\newcommand{\caldop}{\mathcal{D}^{\smallop}}

\newcommand{\calm}{\mathcal{M}}
\newcommand{\calmc}{\calm^{\fC}}

\newcommand{\set}{\mathsf{Set}}

\newcommand{\sset}{\mathsf{sSet}}

\newcommand{\symseq}{\mathsf{SymSeq}}
\newcommand{\symseqc}{\symseq_{\fC}}
\newcommand{\symseqcm}{\symseqc(\calm)}

\newcommand{\alg}{\mathsf{Alg}}
\newcommand{\Alg}{\mathsf{Alg}}

\newcommand{\sigmaop}{\Sigma^{\smallop}}

\newcommand{\sigmaopb}{\Sigma^{\smallop}_{\smallbrb}}
\newcommand{\sigmaopc}{\Sigma^{\smallop}_{\smallbrc}}

\newcommand{\opc}{\mathsf{Op}^{\fC}}
\newcommand{\opcset}{\mathsf{Op}^{\fC}_{\mathsf{Set}}}
\newcommand{\operad}{\mathsf{Operad}}
\newcommand{\operadsigma}{\operad}
\newcommand{\operadsigmac}{\operad^{\fC}}

\newcommand{\operadsigmacm}{\operad^{\fC}({\calm})}

\newcommand{\tree}{\mathsf{Tree}}

\newcommand{\reducedtree}{\mathsf{rTree}}

\newcommand{\pofc}{\Sigma_{\frakC}}
\newcommand{\pofcop}
{\pofc^{\scalebox{.6}{$\mathrm{op}$}}}

\newcommand{\smallprof}[1]
{\raisebox{.05cm}{\scalebox{0.8}{#1}}}

\newcommand{\cjbrbj}
{\smallprof{$\binom{c_j}{[\ub_j]}$}}

\newcommand{\singlecibi}
{\smallprof{$\binom{c_i}{\ub_i}$}}

\newcommand{\ccsingle}
{\smallprof{$\binom{c}{c}$}}

\newcommand{\singledbrb}
{\smallprof{$\binom{d}{[\ub]}$}}
\newcommand{\dub}
{\smallprof{$\binom{d}{\ub}$}}

\newcommand{\duc}
{\smallprof{$\binom{d}{\uc}$}}

\newcommand{\singledbrc}
{\smallprof{$\binom{d}{[\uc]}$}}

\newcommand{\singledempty}
{\smallprof{$\binom{d}{\varnothing}$}}

\newcommand{\smallbr}[1]
{\raisebox{.03cm}{\scalebox{0.5}{#1}}}

\newcommand{\smallbrb}{\smallbr{$[\ub]$}}

\newcommand{\smallbrc}{\smallbr{$[\uc]$}}

\newcommand{\sigmabra}{\Sigma_{\smallbr{$[\ua]$}}}

\newcommand{\sigmabrb}{\Sigma_{\smallbr{$[\ub]$}}}

\newcommand{\sigmabrbj}{\Sigma_{\smallbr{$[\ub_j]$}}}

\newcommand{\sigmabrc}{\Sigma_{\smallbr{$[\uc]$}}}

\newcommand{\sigmabrr}{\Sigma_{\smallbr{$[r]$}}}

\newcommand{\sigmabrbop}{\sigmabrb^{\smallop}}
\newcommand{\sigmabrbjop}{\sigmabrbj^{\smallop}}

\newcommand{\sigmabrcop}{\sigmabrc^{\smallop}}

\newcommand{\sigmabrcopd}{\sigmabrcop \times \{d\}}

\newcommand{\sigmaofc}{\pofc}
\newcommand{\sigmacop}{\pofcop}
\newcommand{\sigmacopc}{\sigmacop \times \fC}
\newcommand{\nofc}{\mathbb{N}(\fC)}

\newcommand{\dbrch}{([\uc];d)}

\newcommand{\andspace}{\qquad\text{and}\qquad}

\newcommand{\vertex}{\mathsf{Vt}}

\renewcommand{\lim}{\mathsf{lim}\,}

\DeclareMathOperator*{\colim}{\mathsf{colim}\,}

\DeclareMathOperator{\Hom}{Hom}
\DeclareMathOperator{\Aut}{Aut}
\DeclareMathOperator{\End}{End}

\DeclareMathOperator{\id}{id}
\DeclareMathOperator{\Id}{Id}
\DeclareMathOperator{\Kan}{\mathsf{Lan}}

\DeclareMathOperator{\Ob}{Ob}

\begin{document}

\title{Relative left properness of colored operads}

\author{Philip Hackney, Marcy Robertson, and Donald Yau}
\address{Matematiska institutionen\\ Stockholms universitet\\ 106 91 Stockholm \\ Sweden}
\email{hackney@math.su.se} 
\address{Department of Mathematics \\ The University of California Los Angeles \\ Los Angeles, CA}\email{marcy.robertson@unimelb.edu.au}
\address{Department of Mathematics \\ The Ohio State University at Newark \\ Newark, OH}
\email{dyau@math.osu.edu}

\begin{abstract}
The category of $\mathfrak{C}$-colored symmetric operads admits a cofibrantly generated model category structure. In this paper, we show that this  model structure satisfies a relative left properness condition, ie that the class of 
weak equivalences between $\Sigma$-cofibrant operads is closed under cobase change along cofibrations. We also provide an example of Dwyer which shows that the model structure on $\fC$-colored symmetric operads is not left proper.
\end{abstract}

\maketitle


\section{Introduction}

Operads are combinatorial devices that encode families of algebras defined by multilinear 
operations and relations. Common examples are the operads $\mathbb{A}$, $\mathbb{C}$ and 
$\mathbb{L}$ whose algebras are associative, associative and commutative, and Lie algebras, 
respectively. Colored operads are a bit more exotic, with what is likely the most famous example being Voronov's ``Swiss-Cheese operad,'' which models the genus-zero moduli spaces that appear in open-closed string theory. Other examples of colored operads\footnote{Colored operads are also sometimes called (symmetric) multicategories in the literature.} encode complicated algebraic structures such as operadic modules, enriched categories, and even categories of operads themselves. The study of model category structures on categories of colored operads has found many recent applications including the rectification of diagrams of operads~\cite{bm07} and the construction of simplicial models for $\infty$-operads~\cite{cmc}.

Our goal in this paper is to further the study of the Quillen model category structure of colored 
operads initiated in ~\cite{rob11,cmc,cav14}. Specifically we are interested in understanding if the category of colored, symmetric operads is \textbf{left proper}; ie we wish to know if weak equivalences between \emph{all} colored, symmetric operads are closed under cobase change along cofibrations.  The main result of this paper is to say that this is not the case, but we give sufficient conditions on a monoidal model category $\calm$ in order for the 
model category structure of $\calm$-enriched, colored, symmetric operads to be \textbf{relatively} left proper, ie for the class of 
weak equivalences between $\Sigma$-cofibrant operads to be closed under cobase change along cofibrations (Theorem \ref{oalg-model-left-proper}).  Recall that in any model category, the class of weak equivalences between cofibrant objects is closed under cobase change along cofibrations. The class of $\Sigma$-cofibrant operads is much larger than the class of cofibrant operads; in particular, this class includes small examples such as the associative operad $\mathbb{A}$.
If one is instead willing to consider the category of \emph{reduced} (or \emph{constant-free}) operads (those satisfying $\sP\binom{c}{\varnothing} = \varnothing$), then Batanin and Berger \cite{bb14} prove a strict left properness result.

The question of (relative) left properness for categories of symmetric operads has many immediate applications.   As an example, left properness makes it easier to identify homotopy pushouts since, in a left proper model category, any pushout along a cofibration is a homotopy pushout.  Relative left properness allows us to make similar statements.

Furthermore, understanding when left properness holds allows us to describe the rectification of homotopy coherent diagrams and weak maps between homotopy $\sO$-algebras, as first proposed by Berger and Moerdijk in \cite[Section 6]{bm07}.  
More explicitly, it is well known that the structure of a model category on the category of $\calm$-enriched operads is important for the study of up to homotopy algebras over an operad such as $\mathbb{A}_{\infty}$-algebras and $\mathbb{E}_{\infty}$-algebras which are respectively associative and commutative ``up to homotopy.'' The deformations of algebraic structures and morphisms between algebraic structures are controlled by up-to-homotopy resolutions of (colored) operads. These resolutions include the W-construction of Boardman and Vogt \cite{bv}, the cobar-bar resolutions of Ginzburg and Kapranov \cite{gk94} and Kontsevich and Soibelman \cite{ks00}, and the Koszul resolutions of Fresse \cite{fre04}. In their paper \cite{bm07}, Berger and Moerdijk show that a coherent theory of up-to-homotopy resolutions of operadic algebras is provided by a Quillen model category structure on $\mathfrak{C}$-colored operads in a general monoidal model category $\calm$. 
(Relative) left properness is one way to establish when these resolutions can be rectified, in the sense of being weakly homotopy equivalent to strict $\sO$-algebras.

\noindent\textbf{Related Work}
To the knowledge of the authors, the idea of relative left properness, and much of the inspiration for this paper, was first established in the thesis of Spitzweck~\cite{spitzweck-thesis} where he considers semimodel structures of categories of operads in general monoidal model categories. Similarly, Dwyer and Hess \cite{dh} and Muro \cite{mur14} established a left properness result which is identical to that of Theorem~\ref{oalg-model-left-proper} for nonsymmetric, monochromatic operads enriched in simplicial sets, respectively, monoidal model categories. Of particular note, Muro's proof requires that his monoidal model categories satisfy weaker conditions than those imposed on the monoidal model categories in this work. The stronger conditions in Theorem~\ref{oalg-model-left-proper} are due to both the extra complexity introduced by the addition of the symmetric group actions and the authors' desire to exhibit the most direct proof of this result which still applies in many situations.

It must also be noted that one could obtain similar results using the techniques of the recent paper of Batanin-Berger~\cite{bb14}; see remark \ref{remark_on_bb}.
The actual definition of relative left properness in \cite{bb14} is slightly different, though morally the same, as that used in Sptizweck~\cite{spitzweck-thesis}, Muro~\cite{mur14}, and this paper and we have made note of similarities in their results and our own throughout this paper. 
Again, the authors of this work have made stronger assumptions on our enriching monoidal model category, as it is our belief that these assumptions allowed for greater clarity in the arguments while still being applicable in most cases of interest.
These assumptions also allow for generalizations to more complicated cases such as relative left properness of dioperads and wheeled properads \cite{hry4} (the latter of which is inaccessible to the Batanin-Berger machinery; see \cite[10.8]{bb14}), which will serve as key components of the authors' larger body of work constructing models for $\infty$-wheeled properads.

\noindent \textbf{Acknowledgments} 
The authors would like to thank Giovanni Caviglia and Kathryn Hess for enlightening discussions and for pointing out errors in earlier drafts of this paper. We would also like to thank Bill Dwyer for allowing us to use his counter-example to left properness of colored operads in Section~\ref{counter-example}. This counter-example was also independently obtained by Caviglia as part of his thesis work. 
Finally, the authors would like to thank the anonymous referee for several insightful comments on an earlier draft of this paper.

\section{Colored Operads and Algebras}
In this section, we briefly recall the definitions of colored operads and algebras over colored operads. 

\subsection{Colors and Profiles}

Throughout, let $(\calm,\otimes,\tensorunit)$ be a closed, symmetric monoidal category with all small colimits. 
Let $\varnothing$ denote the initial object of $\calm$ and $\Hom(X,Y) \in \calm$ the internal hom object.
We will briefly give the necessary definitions and notations regarding colored objects in $\calm.$ A more complete discussion of the following definitions can be found in \cite{jy2} . 

\begin{definition}[Colored Objects]
\label{def:profiles}
Fix a non-empty set  of \textbf{colors}, $\fC$. 
\begin{enumerate}
\item
A \textbf{$\fC$-profile} is a finite sequence of elements in $\fC$,
\[
\uc = (c_1, \ldots, c_m) = c_{[1,m]}
\]
with each $c_i \in \fC$.  If $\fC$ is clear from the context, then we simply say \textbf{profile}. The empty $\fC$-profile is denoted $\varnothing$, which is not to be confused with the initial object in $\calm$.  Write $|\uc|=m$ for the \textbf{length} of a profile $\uc$.
\item
An object in the product category $\prod_{\fC} \calm = \calm^{\fC}$ is called a \textbf{$\fC$-colored object in $\calm$}; similarly a map of $\fC$-colored objects is a map in $\prod_{\fC} \calm$.  A typical $\fC$-colored object $X$ is also written as $\{X_a\}$ with $X_a \in \calm$ for each color $a \in \fC$.
\item Fix $c \in \fC$.  An $X \in \calmc$ is said to be \textbf{concentrated in the color $c$} if $X_d = \varnothing$ for all $c \not= d \in \fC$.
\item
Similarly, fix $c\in\fC$.  For $f : X \to Y \in \calm$ we say that $f$ is said to be \textbf{concentrated in the color $c$} if both $X$ and $Y$ are concentrated in the color $c$.
\end{enumerate}
\end{definition}

Now we are ready to define the colored version of $\Sigma$-objects underlying the category of colored operads. These objects are also sometimes called symmetric sequences, $\Sigma$-modules, or collections in the literature.

\begin{definition}[Colored Symmetric Sequences]
\label{def:colored-sigma-object}
Fix a non-empty set $\fC$.
\begin{enumerate}
\item
If $\ua$ and $\ub$ are $\fC$-profiles, then a \textbf{map} (or \textbf{left permutation}) $\sigma : \ua \to \ub$ is a permutation $\sigma \in \Sigma_{|\ua|}$ such that
\[
\sigma\ua = (a_{\sigma^{-1}(1)}, \ldots , a_{\sigma^{-1}(m)}) = \ub
\]
This necessarily implies $|\ua| = |\ub| = m$.
\item
The \textbf{groupoid of $\fC$-profiles}, which has objects the $\mathfrak{C}$-profiles, and left permutations as the isomorphisms, is denoted by $\pofc$.  The opposite groupoid, $\pofcop,$ is the groupoid of $\fC$-profiles with \textbf{right permutations}
\[
\ua\sigma = (a_{\sigma(1)}, \ldots , a_{\sigma(m)})
\]
as isomorphisms.
\item
The \textbf{orbit} of a profile $\ua$ is denoted by $[\ua]$.  The maximal connected sub-groupoid of $\pofc$ containing $\ua$ is written as $\sigmabra$.  Its objects are the left permutations of $\ua$.  There is an \textbf{orbit decomposition} of $\pofc$
\begin{equation}
\label{pofcdecomp}
\pofc \cong \coprod_{[\ua] \in \pofc} \sigmabra,
\end{equation}
where there is one coproduct summand for each orbit $[\ua]$ of a $\fC$-profile.  \item
Define the diagram category
\begin{equation}
\label{colored-symmtric-sequence}
\symseqcm = \calm^{\sigmacopc},
\end{equation}
whose objects are called \textbf{$\fC$-colored symmetric sequences} or just \textbf{symmetric sequences} when $\fC$ is understood.  The decomposition \eqref{pofcdecomp} implies that there is a decomposition
\begin{equation}
\label{symseqdecomp}
\symseqcm 
\cong 
\prod_{\dbrch \in \sigmacopc} \calm^{\sigmabrcopd},
\end{equation}
where $\sigmabrcopd \cong \sigmaopc$.  
\item
For $X \in \symseqcm$, we write
\begin{equation}
\label{sigmacopd-component}
X\singledbrc \in \calm^{\sigmabrcopd} \cong \calm^{\sigmabrcop}
\end{equation}
for its $\dbrch$-component.  For $(\uc;d) \in \sigmacopc$ (ie $\uc$ is a $\fC$-profile and $d \in \fC$), we write
\begin{equation}
\label{dc-component}
X\duc \in \calm
\end{equation}
for the value of $X$ at $(\uc;d)$.
\item
Write $\nofc$ for the set $\Ob(\sigmacopc)$, ie an element in $\nofc$ is a pair $(\uc;d) \in \sigmacopc$.
\end{enumerate}
\end{definition}

\begin{remark}
\label{soneobject}
In the case where $\fC = \{*\}$, for each integer $n \geq 0$, there is a unique $\fC$-profile of length $n$, usually denoted by $[n]$.  We have $\Sigma_{[n]} = \Sigma_n$, which is just the symmetric group $\Sigma_n$ regarded as a one-object groupoid.  So we have $\nofc = \mathbb{N}$,
\[
\pofc = \coprod_{n \geq 0} \Sigma_n = \Sigma, 
\andspace
\symseqcm = \calm^{\sigmacopc} = \calm^{\sigmaop}.
\]
So one-colored symmetric sequences are symmetric sequences (also known as $\Sigma$-objects and collections) in the usual sense.
\end{remark}

Unless otherwise specified, we will assume that $\fC$ is a \textbf{fixed}, non-empty set of colors.

\subsection{Colored Circle Product}

We define $\fC$-colored operads to be monoids in $\symseqcm$ with respect to the $\fC$-colored circle product. In order to define the latter, we need the following definition.

\begin{definition}[Tensored over a Groupoid]
\label{def:tensorover}
Suppose $\cald$ is a small groupoid, $X \in \calm^{\caldop}$, and $Y \in \calm^{\cald}$.  Define the object $X \otimes_{\cald} Y \in \calm$ as the colimit of the composite
\[
\nicexy{
\cald \ar[r]^-{\cong \Delta} 
& \caldop \times \cald \ar[r]^-{(X,Y)}
& \calm \times \calm \ar[r]^-{\otimes}
& \calm,
}\]
where the first map is the composite of the diagonal map and the isomorphism $\cald \times \cald \cong \caldop \times \cald$.
\end{definition}

We mainly use the construction $\otimes_{\cald}$ when $\cald$ is the finite connected groupoid $\sigmabrc$ for some orbit $[\uc] \in \pofc$.

\begin{convention}
For an object $A \in \calm$, $A^{\otimes 0}$ is taken to mean $\tensorunit$, the $\otimes$-unit in $\calm$.
\end{convention}

\begin{definition}[Colored Circle Product]
\label{def:colored-circle-product}
Suppose $X,Y  \in \symseqcm$, $d \in \fC$, $\uc = (c_1,\ldots,c_m) \in \pofc$, and $[\ub] \in \pofc$ is an orbit.
\begin{enumerate}
\item
Define the object $Y^{\uc} \in \calm^{\pofcop} \cong \prod_{[\ub] \in \pofc} \calm^{\sigmabrbop}$ as having the $[\ub]$-component
\begin{equation}
\label{ytensorc}
Y^{\uc}([\ub]) 
=
\coprod_{\substack{\{[\ub_j] \in \pofc\}_{1 \leq j \leq m} \,\mathrm{s.t.} \\
[\ub] = [(\ub_1,\ldots,\ub_m)]}} 
\Kan^{\sigmabrbop} 
\left[\bigotimes_{j=1}^m Y \cjbrbj\right] 
\in \calm^{\sigmabrbop}.
\end{equation}
The Kan extension in \eqref{ytensorc} is defined as shown:
\[
\nicexy{
\prod_{j=1}^m \sigmabrbjop 
\ar[d]_-{\mathrm{concatenation}} 
\ar[rr]^-{\prod Y \binom{c_j}{-}} 
&&
\calm^{\times m} \ar[d]^-{\otimes}
\\
\sigmabrbop \ar[rr]_-{\Kan^{\sigmabrbop}\left[\bigotimes Y(\scalebox{0.7}{\vdots})\right]}^-{\mathrm{left ~Kan~ extension}} 
&& \calm.
}\]
\item
Since we consider left permutations of $\uc$ in \eqref{ytensorc}, we obtain $Y^{[\uc]} \in \calm^{\pofcop \times \sigmabrc} \cong \prod_{[\ub] \in \pofc} \calm^{\sigmabrbop \times \sigmabrc}$ with components
\begin{equation}
\label{tensorbracket}
Y^{[\uc]}([\ub]) \in \calm^{\sigmabrbop \times \sigmabrc}.
\end{equation}
\item
Using the product decomposition \eqref{symseqdecomp} of $\symseqcm$, 
the \textbf{$\fC$-colored circle product} $X \circ Y \in \symseqcm$ is defined to have components
\begin{equation}
\label{xcircley}
(X \circ Y)\singledbrb 
= \coprod_{[\uc] \in \pofc} 
X\singledbrc \otimes_{\sigmabrc} Y^{[\uc]}([\ub]) \in \calm^{\sigmaopb \times \{d\}},
\end{equation}
where the coproduct is indexed by all the orbits in $\pofc$, as $d$ runs through $\fC$ and $[\ub]$ runs through all the orbits in $\pofc$.  The construction $\otimes_{\sigmabrc}$ is as defined in Definition \ref{def:tensorover}.
\end{enumerate}
\end{definition}

\begin{remark}
In the one-colored case (ie $\fC = \{*\}$), the $\fC$-colored circle product is equivalent to the circle product of $\Sigma$-objects in \cite{rezk} (2.2.3). An anonymous referee made the authors aware that the idea to first define the circle product through Day's convolution belongs to G.M. Kelly \cite{kelly}.
\end{remark}

The following observation is the colored version of \cite{harper-jpaa} (4.13).

\begin{proposition}
\label{circle-product-monoidal}
With respect to $\circ$, $\symseqcm$ is a monoidal category.
\end{proposition}

\begin{remark}
We consider $\calmc$ as a subcategory of $\calm^{\nofc}$ via the inclusion 
\begin{align*}
	\fC &\rightarrow \nofc \\
	c & \mapsto \binom{c}{\varnothing}.
\end{align*}
We use this to consider $\sO \comp - $ as a functor with domain $\calmc$ in example \ref{monad example}.
\end{remark}

\subsection{Colored Operads as Monoids}\label{monoids}
In the previous section we show that the category of $\mathfrak{C}$-colored operads is a category of monoids ``with many objects.'' We make this explicit below. 

\begin{definition}
\label{def:colored-operad}
For a non-empty set $\fC$ of colors, denote by $\operadsigmacm$ or $\operadsigmac$, when $\calm$ is understood, the category of monoids \cite{maclane} (VII.3) in the monoidal category $(\symseqcm, \comp)$.  An object in $\operadsigmac$ is called a \textbf{$\fC$-colored operad} in $\calm$.
We write $\varnothing_{\fC}$ for the initial object in $\operadsigmac$.
\end{definition}

\begin{remark}
Unpacking Definition \ref{def:colored-operad}, a $\fC$-colored operad is equivalent to a triple $(\sO, \gamma, u)$ consisting of:
\begin{itemize}
\item
$\sO \in \symseqcm$,
\item
a \textbf{$\fC$-colored unit} map
\[
\nicexy{\tensorunit \ar[r]^-{u_c} & \sO\ccsingle} \in \calm
\]
for each color $c \in \fC$, and
\item
\textbf{operadic composition}
\begin{equation}
\label{operadic-comp}
\nicexy{
\sO\duc \otimes \bigotimes_{i=1}^m \sO\singlecibi \ar[r]^-{\gamma} 
& \sO\dub \in \calm
}
\end{equation}
for all $d \in \fC$, $\uc  = (c_1,\ldots,c_m) \in \sigmaofc$ with $m \geq 1$, and $\ub_i \in \sigmaofc$, where $\ub = (\ub_1,\ldots,\ub_m)$.
\end{itemize}
The triple $(\sO,\gamma,u)$ is required to satisfy the obvious associativity, unity, and equivariance axioms, the details of which can be found in \cite{jy2} (11.14).   The detailed axioms in the one-colored case can also be found in \cite{may97}.  This way of expressing a $\fC$-colored operad is close to the way an operad was defined in \cite{may72}.
\end{remark}

\begin{remark}
In the case $\fC = \{*\}$, write $\operadsigma$ for $\operadsigmac$. Objects of this category are called \textbf{$1$-colored operads} or \textbf{monochromatic operads}.  In this case we write $\sO(n)$ for the $([n]; *)$-component of $\sO \in \operadsigma$, where $[n]$ is the orbit of the $\{*\}$-profile consisting of $n$ copies of $*$ (this orbit has only one object).  Our notion of a $1$-colored operad agrees with the notion of an operad in, eg \cite{may97} and \cite{harper-jpaa}.  
Note that even for $1$-colored operads, our definition is slightly more general than the one in \cite{mss} (II.1.2) because in our definition, the $0$-component $\sO(0)$ corresponds to the empty profile, $\{*\}$.  In general, the purpose of the $0$-component (whether in the one-colored or the general colored cases) is to encode units in $\sO$-algebras.  Also note that in \cite{may72}, where an operad was first defined in the topological setting, the $0$-component was required to be a point.
\end{remark}

\begin{definition}
Suppose $n \geq 0$.  A $\fC$-colored symmetric sequence $X$ is said to be \textbf{concentrated in arity $n$} if
\[
|\uc| \not= n 
\quad \Longrightarrow\quad 
X\duc = \varnothing ~ \text{for all $d \in \fC$.}
\] 
\end{definition}

\begin{example}\label{monad example}
\begin{enumerate}
\item
A $\fC$-colored symmetric sequence concentrated in arity $0$ is precisely a $\fC$-colored object.  In the $\fC$-colored circle product $X \comp Y$ \eqref{xcircley}, if $Y$ is concentrated in arity $0$, then so is $X \comp Y$ because, by \eqref{ytensorc}, 
\[
\ub \not= \varnothing \quad \Longrightarrow \quad Y^{\uc}([\ub]) = \varnothing
\]
for all $\uc$.
In other words, there is a lift
\[
        \nicexy{
                \calmc \ar@{.>}[r] \ar@{->}[d] & \calmc \ar@{->}[d]  \\
                \symseqcm \ar@{->}[r]^-{\sO \comp -}  & \symseqcm.
        }
\]
So if $\sO$ is a $\fC$-colored operad, then the functor
\begin{equation}
\label{o-comp-monad}
\sO \comp - : \calmc \to \calmc
\end{equation}
defines a monad \cite{maclane} (VI.1)  whose monadic multiplication and unit are induced by the multiplication $\sO \comp \sO \to \sO$ and the unit $\varnothing_{\fC} \to \sO$, respectively.
\item
A $\fC$-colored operad $\sO$ concentrated in arity $1$ is exactly an $\calm$-enriched category with object set $\fC$.  In this case, the non-trivial operadic compositions correspond to the categorical compositions.  Restricting further to the $1$-colored case $(\fC = \{*\})$, a $1$-colored operad concentrated in arity $1$ is precisely a monoid in $\calm$.
\end{enumerate}
\end{example}




\subsection{Algebras over Colored Operads} 
 The category of representations over an operad $\sO$, is referred to, for classical reasons, as the category of \textbf{algebras over an operad}.  

\begin{definition}
\label{colored-operad-algebra}
Suppose $\sO$ is a $\fC$-colored operad.  The category of algebras over the monad \cite{maclane} (VI.2)
\[
\sO \comp - : \calmc \to \calmc
\]
in \eqref{o-comp-monad} is denoted by $\alg(\sO; \calm)$ or simply $\alg(\sO)$ when $\calm$ is understood.  Objects of $\alg(\sO)$ are called \textbf{$\sO$-algebras} (in $\calm$). 
\end{definition}

\begin{definition}
\label{def:asubc}
Suppose $A = \{A_c\}_{c\in \fC} \in \calm^{\fC}$ is a $\fC$-colored object.  For $\uc = (c_1,\ldots,c_n) \in \pofc$ with associated orbit $[\uc]$, define the object
\begin{equation}
\label{asubc}
A_{\uc} = \bigotimes_{i=1}^n A_{c_i} = A_{c_1} \otimes \cdots \otimes A_{c_n} \in \calm
\end{equation}
and the diagram $A_{\smallbrc} \in \calm^{\sigmabrc}$ with values
\begin{equation}
\label{asubbrc}
A_{\smallbrc}(\uc') = A_{\uc'}
\end{equation}
for each $\uc' \in [\uc]$.  All the structure maps in the diagram $A_{\smallbrc}$ are given by permuting the factors in $A_{\uc}$.
\end{definition}

\begin{remark}[Unwrapping $\sO$-Algebras]
From the definition of the monad $\sO \comp -$, an $\sO$-algebra $A$ has a structure map $\mu : \sO \comp A \to A \in \calmc$.  For each color $d \in \fC$, the $d$-colored entry of $\sO \comp A$ is
\begin{equation}
\label{o-comp-a-d}
(\sO \comp A)_d = \coprod_{[\uc] \in \sigmaofc} 
\sO\singledbrc \otimes_{\sigmabrc} A_{\smallbrc}.
\end{equation}
So the $d$-colored entry of the structure map $\mu$ consists of maps
\[
\nicexy{
\sO\singledbrc \otimes_{\sigmabrc} A_{\smallbrc}
\ar[r]^-{\mu}
& A_d \in \calm
}\]
for all orbits $[\uc] \in \sigmaofc$.  The $\otimes_{\sigmabrc}$ here means that we can unpack $\mu$ further into maps
\begin{equation}
\label{algebra-map-unpack}
\nicexy{
\sO\duc \otimes A_{\uc}
\ar[r]^-{\mu}
& A_d \in \calm
}
\end{equation}
for all $d \in \fC$ and all objects $\uc \in \sigmaofc$.  Then an $\sO$-algebra is equivalent to a $\fC$-colored object $A$ together with structure maps \eqref{algebra-map-unpack} that are associative, unital, and equivariant in an appropriate sense, the details of which can be found in \cite{jy2} (13.37).  The detailed axioms in the $1$-colored case can also be found in \cite{may97}.  Note that when $\uc = \varnothing$, the map \eqref{algebra-map-unpack} takes the form
\begin{equation}
\label{d-colored-units}
\nicexy{\sO \singledempty \ar[r]^-{\mu} & A_d}
\end{equation}
for $d \in \fC$.  In practice this $0$-component of the structure map gives $A$ the structure of $d$-colored units.  For example, in a unital associative algebra, the unit arises from the $0$-component of the structure map.
\end{remark}

\begin{remark}\label{endomorphism}
The \textbf{$\fC$-colored endomorphism operad}, $\End(A)$ is defined by
$$ \End\duc =\Hom_{\calm}(A_{\uc},A_{d}).$$ It is an elementary exercise to check that, for an $\fC$-colored operad $\sO$, an $\sO$-algebra $A$ is equivalent to a map of $\fC$-colored operads 
\[\nicexy{\sO \ar[r]^-{\mu} & \End(A).}\]
\end{remark}

Some important examples of colored operads and their algebras follow.

\begin{example}[Free Operadic Algebras]
\label{ex:free-algebra}
Fix a $\fC$-colored operad $\sO$.  There is an adjoint pair
\begin{equation}
\label{free-algebra-adjoint}
\nicexy{
\calmc \ar@<2pt>[r]^-{\sO \comp -} 
& \alg(\sO) \ar@<2pt>[l]
}
\end{equation}
in which the right adjoint is the forgetful functor.  The left adjoint takes a $\fC$-colored object $A$ to the object $\sO \comp A$ which has the canonical structure of an $\sO$-algebra, called the \textbf{free $\sO$-algebra of $A$}.  In particular, free $\sO$-algebras always exist.
\end{example}

\begin{example}
If $\sO$ is an $\calm$-enriched category, then the category of $\sO$-algebras is the $\calm$-enriched functor category $[\sO,\calm]$.
 \end{example}

\begin{example}[$\fC$-Colored Operads as Operadic Algebras]
\label{ex:operad-of-operad} Recall that $\nofc = \Ob(\pofcop \times \fC)$. 
For each non-empty set of colors $\fC$, there exists an $\nofc$-colored operad $\opc$ and an isomorphism
\begin{equation}
\label{colored-op-algebras}
\operadsigmac \cong \alg(\opc).
\end{equation}
So $\fC$-colored operads are equivalent to algebras over the $\nofc$-colored operad $\opc$.  This is a special case of \cite{jy2} (14.4), which describes any category of generalized props (of which $\operadsigmac$ is an example) as a category of algebras over some colored operad; in the case $\fC = \{ * \}$ this construction appears in \cite[1.5.6]{bm07}.  
As mentioned in Example \ref{ex:free-algebra}, it follows that \textbf{free $\fC$-colored operads} ($=$ free $\opc$-algebras) always exist.  
The construction of $\opc$ begins with an $\nofc$-colored operad $\opcset$ in the symmetric monoidal category of sets and Cartesian products.  
There is a strong symmetric monoidal functor
\begin{equation}\label{set to calm}
\set \to \calm, \quad S \longmapsto \coprod_S \tensorunit.
\end{equation}
The colored operad $\opc$ is the entry-wise image of $\opcset$ under this strong symmetric monoidal functor.  Therefore, if $\calm$ has a model structure in which $\tensorunit$ is cofibrant, then $\opc$ is entry-wise cofibrant.  
In fact, when $\tensorunit$ is cofibrant, a careful inspection of $\opc$ shows that its underlying symmetric sequence is cofibrant in $\symseqcm$.  This is a key example for us, and we will elaborate on it more later.
\end{example}

\subsection{Limits and Colimits of Colored Operadic Algebras} 
Limits of $\alg(\sO)$ are taken in the underlying category of colored objects $\calmc$ via the free-forgetful adjoint pair
\[
\nicexy{
\calmc \ar@<2pt>[r]^-{\sO \comp -} 
& \alg(\sO) \ar@<2pt>[l]
}\]
in \eqref{free-algebra-adjoint} for a $\fC$-colored operad.  The following observation is the colored version of a well known result (see, for example ~\cite[Prop 2.3.5]{rezk}, \cite[5.15]{harper-jpaa}, or the closely related \cite[Proposition II.7.2]{ekmm}).

\begin{proposition}
\label{algebra-bicomplete}
Suppose $\sO$ is a $\fC$-colored operad. Then the category $\alg(\sO)$ has all small limits and colimits, with reflexive coequalizers and filtered colimits preserved and created by the forgetful functor $\alg(\sO) \to \calmc$.
\end{proposition}

\subsection{Model Structure on Colored Operadic Algebras} 
In this section we will assume that our cocomplete, closed, symmetric monoidal category $\calm$ comes with a compatible cofibrantly generated Quillen model category structure, ie we assume that $\calm$ is a \emph{monoidal model category} \cite[Def 3.1]{ss} with cofibrant tensor unit.

The category of $\fC$-colored objects, $\calmc$, admits a cofibrantly generated model category structure where weak equivalences, fibrations, and cofibrations are defined entrywise, as described in \cite{hirschhorn} (11.1.10).  
In this model category a generating cofibration in $\calmc = \prod_{\fC} \calm$ (ie a map in $\sI$) is a generating cofibration of $\calm$, concentrated in one entry.  Similarly, the set of generating acyclic cofibrations is $\sJ \times \fC$.  In addition, the properties of being simplicial, or proper, are inherited from $\calm$.

A functor $F$ between two symmetric monoidal categories is called \textbf{symmetric monoidal} if there is a unit $\tensorunit \to F(\tensorunit)$ and a binatural transformation \[ F(-) \otimes F(-) \Rightarrow F(- \otimes -)\] satisfying unit, associativity, and symmetry conditions \cite{maclane}.

\begin{definition} 
We say that $\calm$ admits \textbf{functorial path data} if there exist a symmetric monoidal
functor $Path$ on $\calm$ and monoidal natural transformations 
\begin{align*} s &: \Id \Rightarrow Path \\
d_0,d_1 &: Path \Rightarrow \Id
\end{align*} 
so that for any fibrant $X$ in $\calm$
\[\nicexy{ X \ar[r]^-{s} & Path(X) \ar[r]^-{d_0\times d_1} & X \times X}\] 
is a path object (ie $s$ is a weak equivalence and $d_0 \times d_1$ is a fibration). 
\end{definition} 

\begin{remark} 
The definition of functorial path data is adapted from Fresse \cite[Fact 5.3]{fresse}. As a particular example, Fresse showed that functorial path data exists if $\calm$ is the category of chain complexes over a ring of characteristic $0$ or the category of simplicial modules. 
\end{remark} 

One way to check if $\calm$ admits functorial path data is to check if $\calm$ admits an interval object defined as follows.  
 
\begin{definition} 
We say that $\calm$ admits a \textbf{cocommutative, coassociative coalgebra interval} $J$ if the fold map $\tensorunit \sqcup \tensorunit \rightarrow \tensorunit$ can be factored as
\[
\nicexy{ \tensorunit\sqcup \tensorunit\ar[r]^{\alpha} & J \ar[r]^{\beta} &\tensorunit}
\] in which $\alpha$ is a cofibration, $\beta$ is a weak equivalence, $J$ is a coassociative cocommutative comonoid in $\calm$, and $\alpha$ and $\beta$ are both maps of comonoids.
\end{definition}

For example, the categories of compactly generated spaces and simplicial sets admit such cocommutative coalgebra intervals. The category of unbounded chain complexes over a ring which is \emph{not} characteristic $0$ admits an interval which is coassociative, but not cocommutative. 

\begin{lemma}\emph{\cite[3.10]{jy1}} If $\calm$ admits a coassociative, cocommutative coalgebra interval and $\tensorunit$ is cofibrant, then $\calm$ admits functorial path data. 
\end{lemma} 

\begin{definition}
A \textbf{symmetric monoidal fibrant replacement functor} is a functor 
$f: \calm \to \calm$
together with a natural transformation $r: \Id \Rightarrow f$ such that
\begin{itemize}
		\item $r_X : X \to f(X)$ is a fibrant replacement for each object $X$,
	\item $f$ is a symmetric monoidal functor, and
	\item for every $X$ and $Y$ in $\calm$ the following diagram commutes
\[\nicexy{ X\otimes Y \ar[d]^-{r_X\otimes r_Y}\ar[r]^-{r_{X\otimes Y}} & f(X\otimes Y)\\
fX\otimes fY. \ar[ur] &}\]
\end{itemize}

\end{definition}

Throughout this paper, we will want our monoidal model category $\calm$ to satisfy a number of conditions, as we want $\calm$ to have a symmetric monoidal fibrant replacement functor. To simplify the listing of these conditions, we make the following definition. 

\begin{definition} 
A monoidal model category $\calm$ is called \textbf{nice} if 
\begin{itemize}

\item $\calm$ is strongly cofibrantly generated, ie the domain of each generating (acyclic) cofibration is small with respect to the entire category;

\item
there is a symmetric monoidal fibrant replacement functor;
\item
there is functorial path data;
\item
every object is cofibrant;
\item
weak equivalences are closed under filtered colimits.
\end{itemize}
\end{definition} 

Examples of nice monoidal model categories are $\sset$, $\mathbb{Z}$-graded chain complexes in characteristic zero, and simplicial presheaves.

\begin{remark} 
The definition of a nice monoidal model category automatically implies that our monoidal model categories are what are called ``strongly h-monoidal'' in Batanin-Berger~\cite[1.8; 2.5]{bb14} and that our monoidal model categories satisfy the monoid axiom of Schwede and Shipley \cite[3.3]{ss} which also makes an appearance in the work of Muro~\cite{mur14}. 
\end{remark}

The following is a restricted version of Theorem 2.1 \cite{bm07} and is a colored operad analogue of \cite{jy1} (3.11), which dealt with the more complicated case of colored props.

\begin{theorem}
\label{oalg-model}
Suppose $\calm$ is a nice monoidal model category and that  $\sO$ is a $\fC$-colored operad in $\calm$.  
Then $\alg(\sO)$ admits a strongly cofibrantly generated model category structure, in which:
\begin{itemize}
\item
fibrations and weak equivalences are created in $\calmc$, and
\item
the set of generating (acyclic) cofibrations is $\sO \comp \sI$ (resp., $\sO \comp \sJ$), where $\sI$ (resp., $\sJ$) is the set of generating (acyclic) cofibrations in $\calmc$.
\end{itemize}
\end{theorem}

\begin{example} The category of simplicial sets, $\sset$, is a Cartesian closed, cofibrantly generated, monoidal model category that admits a coassociative, cocommutative interval. As a symmetric monoidal fibrant replacement functor, we can choose either the $Ex^{\infty}$ functor or the singular chain complex of the geometric realization functor, since both are product-preserving. Similarly, the category of $\mathbb{Z}$-graded chain complexes over a field $\mathbb{K}$ with the projective model structure \cite[Chapter 2]{hovey} satisfies the conditions of Theorem \ref{oalg-model}.
\end{example} 

\begin{corollary}\label{oalg-model-operad} If $\calm$ is a nice monoidal model category, then $\Alg(\opc) \cong \operadsigmac$ admits a cofibrantly generated model structure. 
\end{corollary}

\begin{definition}\label{def:sigma-cof}
\begin{itemize}

\item The \textbf{fibrant} $\mathfrak{C}$-colored operads are those which are locally fibrant, ie $\sP\binom{d}{\uc}$ is fibrant in $\calm$ for all profiles $(\uc; d).$ 

\item A $\mathfrak{C}$-colored operad is called \textbf{$\Sigma$-cofibrant} if $\sP$ is cofibrant as an object in $\symseqcm = \calm^{\sigmacopc}.$
\end{itemize}

\end{definition}

Every cofibrant operad is, in particular, $\Sigma$-cofibrant~\cite[Proposition 4.3]{bm06}.

\begin{example}
	The associative operad $\mathbb{A}$ is the prototypical $\Sigma$-cofibrant operad which is not cofibrant.
	In $\sset$, the commutative operad $\mathbb{C}$ is neither $\Sigma$-cofibrant nor cofibrant.
\end{example}

\section{Relative Left Properness of Operads with Fixed Colors}

In this section, we show that the model category structure of Corollary~\ref{oalg-model-operad} satisfies a property close to that of left properness, which we will refer to as \textbf{relative} left properness.

\begin{definition}\label{relative definition}
The model category $\operadsigmac$ is called left proper \textbf{relative to the class of $\Sigma$-cofibrant operads} if pushouts by cofibrations preserve weak equivalences whose domain and codomain are $\Sigma$-cofibrant.   
\end{definition}

\subsection{The Pushout Filtration}

Relative left properness of $\operadsigmac$ comes down to a study of pushouts of $\fC$-colored operads where one of the defining maps is a free morphism of free operads (Lemma \ref{opc-jt}). To perform this analysis, we make use of the language of colored, planar trees such as those in \cite[5.8]{bm03}, \cite{gk94} or \cite[Section 3]{bm06}. The following definition comes from Chapter 3~\cite{yau_op}. 

\begin{definition}A \textbf{rooted $n$-tree} is a non-empty, finite, connected, directed graph with no directed cycles in which:  

\begin{enumerate} 

\item there are $n$ distinguished vertices, called inputs, each with exactly one out-going edge and no incoming edges;

\item there is a distinguished vertex that is not an input, called the root, with exactly one incoming edge and no outgoing edges;

\item each vertex away from the set of inputs and the root has exactly one outgoing edge.

\end{enumerate} 

A \textbf{planar rooted tree} is a rooted tree in which the set of incoming edges at each vertex is equipped with a linear ordering.

\end{definition}

\begin{remark} 
For a planar rooted tree $T$, we write $in(T)$ for the set of its input edges. Since $T$ is planar, the input edges (or leaves) have a linear order and we write $\lambda(T)$ for the set of all such orderings \[\nicexy{\{1,...,n\}\ar[r]& in(T)}\] where $n=|in(T)|$. It is fairly easy to check that one can identify the set of all linear orderings of the input edges of $T$, $\lambda(T)$, with the group of permutations $\Sigma_{n}$. 
\end{remark} 

\begin{definition} Let $A\in\symseqcm$ (\ref{def:colored-sigma-object}) and suppose $m \geq 1$, and $t \in \nofc$, $s_j \in \nofc$ for $1 \leq j \leq m$.
\begin{enumerate}
\item
Denote by $\tree(t)$ the groupoid of directed, planar, rooted, $\fC$-colored trees in which the input-output profile is given by $t$. The morphisms in $\tree(t)$ are non-planar isomorphisms of $\fC$-colored trees.

\item
Denote by $\tree(\{s_j\}_1^m;t)$ the groupoid of pairs $(T,\sds)$ such that
\begin{itemize}
\item
$T \in \tree(t)$, and 
\item
$\sds \subseteq \vertex(T)$ such that the set of vertex profiles in $\sds$ is the set $\{s_j\}_1^m$.
\end{itemize}
Vertices in $\sds$ are called \textbf{distinguished vertices}.  Vertices in the complement
\[
\sn(T) \equiv \vertex(T) \setminus \sds
\]
are called \textbf{normal vertices}. Isomorphisms of $\tree(\{s_j\};t)$ are isomorphisms of $\fC$-colored trees which preserve the distinguished vertices and colorings of edges. 

\item 
A pair $(T,\sds)\in\tree(\{s_j\};t)$ is said to be \textbf{well-marked} if every flag of a distinguished vertex is part of an internal edge whose other end vertex is normal.

\item
A pair $(T,\sds) \in \tree(\{s_j\};t)$ is said to be \textbf{reduced} if it is well-marked and there are no adjacent normal vertices, ie every vertex adjacent to a normal vertex is distinguished.  The groupoid of such reduced trees is denoted by $\reducedtree(\{s_j\};t)$.

\item
Given a vertex $u$ in a tree $T$, write $A(u)$ for the component of the symmetric sequence $A$ corresponding to the profiles of $u$.  In other words, if the profiles of $u$ are $(\uc;d) \in \nofc$, then $A(u) = A\duc$.  We also say that $A(u)$ is a \textbf{decoration} of $u$ by $A$ and that $u$ is \textbf{$A$-decorated}.  A tree with each vertex decorated by $A$ is said to be \textbf{$A$-decorated}.

\end{enumerate}
\end{definition}

\begin{definition}\label{aut-to-sigma}
Suppose that $f: H\to G$ is a homomorphism of groups.
Then there is an adjoint pair
\[
	(-) \cdot_H G \,\,\colon\,\, \calm^{H^{op}} \rightleftarrows \calm^{G^{op}} \,\,\colon\,\, f^*;
\]
this adjoint pair is actually a Quillen adjunction \cite[2.5.1]{bm06}.
If $f$ is a subgroup inclusion and $X\in \calm$ is an object with a right $H$ action (ie $X\in \calm^{H^{op}}$), we have
\[
	X \cdot_H G \cong \coprod_{G/H} X
\]
where the coproduct is indexed over the cosets of $H$ in $G$.
\end{definition}

The following definition appears in \cite{harper-jpaa} (7.10).

\begin{definition}[$Q$-Construction]
\label{one-colored-q}
Suppose there is a map $i : X \to Y \in \calm$. 
The object $Q^{t}_q \in \calm^{\Sigma_t}$ is given as follows.
\begin{itemize}
\item
$Q^{t}_0 = X^{\otimes t}$.
\item
$Q^{t}_t = Y^{\otimes t}$.
\item
For $0 < q < t$ there is a pushout in $\calm^{\Sigma_t}$:
\begin{equation}
\label{inductive-q-one-colored}
\nicexy{
\left[ X^{\otimes (t-q)} 
\otimes Q_{q-1}^{q}\right] 
\cdot_{\Sigma_{t-q} \times \Sigma_q} \Sigma_t
\ar[d]_-{(\id,i_*)} \ar[r] \pushout
&
Q^{t}_{q-1} \ar[d]
\\
\left[ X^{\otimes (t-q)} 
\otimes 
Y^{\otimes q}\right]
\cdot_{\Sigma_{t-q} \times \Sigma_q} \Sigma_t 
\ar[r] 
&
Q^{t}_q.
}
\end{equation}
\end{itemize}
\end{definition}

\begin{lemma}
\label{opc-jt}
For $A \in \operadsigmac$ and a map $i : X \to Y$ in $\calm$, regarded as a map in $\calm^{\nofc}$ concentrated in the $s$-entry for some $s \in \nofc$, consider a pushout
\[
\nicexy{
\opc \comp X \ar[d]_{i_*} \ar[r]^-{f} \pushout
& A \ar[d]^-{h}
\\
\opc \comp Y \ar[r]
& A_{\infty}
}\]
in $\operadsigmac$. Then for a fixed orbit $[r]$, with $r \in \nofc$, the $[r]$-entry of the map $h$ is a countable composition
\[
\nicexy{
A([r]) = A_0([r]) \ar[r]^-{h_1}
& A_1([r]) \ar[r]^-{h_2}
& A_2([r]) \ar[r]^-{h_3}
& \cdots \ar[r]
& A_\infty([r]),
}\]
where for $k \geq 1$ the $h_k$ are inductively defined as the following pushout in $\calm^{\Sigma_{[r]}}$
\begin{equation}
\label{opc-jt-pushout}
\nicexy@R+10pt{
\coprod_{[T,\sds]}
\left\{\left[ \bigotimes_{u \in \sn(T)} A(u)\right] 
\otimes Q^k_{k-1}\right\} 
\cdot_{Aut(T,\sds)} \Sigma_{[r]}
\ar[d]_-{\amalg (\id \otimes i^{\boxprod k}) \otimes_{Aut(T,\sds)} \id} 
\ar[r]^-{f^{k-1}_*} \pushout
& 
A_{k-1}([r]) \ar[d]^-{h_k}
\\
\coprod_{[T,\sds]}
\Bigl\{
\underbrace{
\left[ \textstyle{\bigotimes_{u \in \sn(T)} A(u)}\right] 
\otimes Y^{\otimes k}}_{\text{normal/dist. vertex decorations}}
\Bigr\}
\cdot_{Aut(T,\sds)} 
\underbrace{\Sigma_{[r]}}_{\text{input labelling}}
\ar[r]_-{\xi_{k}} 
& 
A_k([r]).
}
\end{equation}
In this pushout:
\begin{enumerate}
\item
The top horizontal map $f^{k-1}_*$ is induced by $f$ and the operad structure map of $A$.
\item
Each coproduct on the left is indexed by the set of weak isomorphism classes of reduced trees $(T,\sds)$ such that:
\begin{itemize}
\item
the input profile of $T$ is in the orbit $[r]$, and 
\item
$\sds$ consists of $k$ distinguished vertices, all with profiles in the orbit $[s]$.
\end{itemize}
\end{enumerate}
\end{lemma}

\begin{proof} This theorem is a special case of Prop 4.3.16 in \cite{yw15} by taking $\mathsf{O}=\opc$; we sketch the proof.
For each $r \in \nofc$, \emph{define}
\[
B([r]) = \colim_k A_k([r]).
\]
Then $B$ has a canonical $\fC$-colored operad structure given as follows.
\begin{itemize}
\item
Its colored units are those of $A$; ie $\tensorunit \to A\ccsingle \to B\ccsingle$ for each $c\in\fC$.

\item
The operadic $\circ_i$ compositions are given by grafting of reduced trees, where the colored operad structure of $A$ is used to bring the grafted tree to a reduced one if necessary.

\item
Its equivariant structure is given by the factors $\Sigma_{|in(T)|}$.

\end{itemize}

The operad map $A \to B$ is induced by $A_0 \to B$.  The map $Y \to B$ is induced by $\xi_1$ (for the $s$-corolla whose only vertex is distinguished) and $A_1 \to B$.  That $B$ is the pushout $A_\infty$ follows from its inductive definition.
\end{proof}

For any finite group $G$, the category of $G$-objects, $\calm^{G}$, has a natural structure of a cofibrantly generated model category where weak equivalences and fibrations are defined entrywise, as described in \cite{hirschhorn} (11.1.10).  
In this model category a generating (acyclic) cofibration is a $G$-equivariant (acyclic) cofibration in the category of $\calm$-objects with $G$-action.  
Because it will be important to keep track of which group we are working with, we will denoted these sets of generating cofibrations by $\sI[G]$ and generating acyclic cofibrations by $\sJ[G]$.

The following lemma, due to Berger-Moerdijk~\cite[Lemma 5.10]{bm03} and Spitzweck~\cite[Lemma 4]{spitzweck-thesis}, gives an equivariant version of the pushout product axiom.

\begin{lemma}\label{prop:equivariant-pushout-product} 
Let $G$ and $\Gamma$ be finite groups with $\Gamma$ acting from the right on $G$. For any $\Gamma$-cofibration $i:X \rightarrow Y$ and any map of right $G\rtimes \Gamma$-objects $A\rightarrow B$ whose underlying map is a cofibration in a nice monoidal model category $\calm$, the induced map $$
X\otimes B\coprod_{X\otimes A} Y\otimes A \longrightarrow Y\otimes B
$$
is a $G\rtimes\Gamma$-cofibration, where $G\rtimes\Gamma$ acts on $Y\otimes B$ by $(y\otimes b)^{(g,\gamma)} = y^{\gamma}\otimes b^{(g,\gamma)}.$
\end{lemma}

In practice, $\Gamma$ will be the symmetric group acting on the inputs of a tree $T$ in $\reducedtree$. 

\begin{lemma}\label{equivariant-pushout-product}
In the context of Lemma \ref{opc-jt}, suppose:
\begin{itemize}
\item
$\calm$ is a nice monoidal model category;
\item
$i : X \to Y \in \calm$ is a cofibration 
\item $A$ is a $\Sigma$-cofibrant operad. 
\end{itemize}
Then each map
\[
\nicexy{
\left[ \bigotimes_{u \in \sn(T)} A(u)\right] 
\otimes Q^k_{k-1}
\ar[d]|-{\id \otimes i^{\boxprod k}}
\\
\left[ \bigotimes_{u \in \sn(T)} A(u)\right] 
\otimes Y^{\otimes k}
}\]
is an $\Aut(T,\sds)$-cofibration.
\end{lemma}

\begin{proof}
As in \cite{bm03} Lemma 5.9, each $(T,\sds)$ has a grafting decomposition as
\[
(T,\sds) = t_n\bigl((T_1,\sds_1), \ldots , (T_n,\sds_n)\bigr),
\]
where 
\begin{itemize}
\item
$t_n$ is the $n$-corolla;
\item
$\sds = \sds_1 \amalg \cdots \amalg \sds_n$ if the top vertex is not distinguished and $\sds =  \sds_1 \amalg \cdots \amalg \sds_n \amalg t_n$ if the top vertex is distinguished.
\end{itemize}
Let 
\[ (T_{j_1}, \sds_{j_1}), \dots, (T_{j_k}, \sds_{j_k}) \in \{(T_1,\sds_1), \dots, (T_n,\sds_n)\} \]
be such that each $(T_\ell,\sds_\ell)$ is isomorphic to exactly one $(T_{j_i}, \sds_{j_i})$,
and let
\[
	n_i = \left|\left\{  (T_\ell,\sds_\ell) \, \middle| \, (T_\ell,\sds_\ell) \cong  (T_{j_i}, \sds_{j_i}) \right\} \right|.
\]
There is a decomposition of the automorphism group,
\[
\Aut(T,\sds) \cong 
\Bigl( 
\underbrace{\prod_{i=1}^k \Aut(T_{j_i}, \sds_{j_i})^{\times n_i}}_{G}
\Bigr)
\rtimes
\Bigl(\underbrace{\prod_{i=1}^k \Sigma_{n_i}}_{\Gamma}\Bigr),
\]
where each $n_i \geq 1$ and $n_1 + \cdots + n_k = n$.
\begin{enumerate}
\item
The map $i^{\boxprod k}$ is a cofibration in $\calm$ by the pushout-product axiom.  Furthermore, it has a right $\Aut(T,\sds)$-action (ie a $G \rtimes \Gamma$-action) because isomorphisms preserve distinguished vertices.

\item  
Since $A(r)$ is $\Gamma$-cofibrant (where $r$ is the vertex at the root) and $\Gamma$ acts on $\bigotimes_{n(T) \setminus r} A(u)$ by permuting tensor factors, we know that $\bigotimes_{n(T)} A(u)$ is $\Gamma$-cofibrant.


\end{enumerate}
These two facts and Lemma~\ref{prop:equivariant-pushout-product} together imply that
\[
\id \otimes i^{\boxprod k} =
\left[\varnothing \to \bigotimes A(u)\right]
\boxprod i^{\boxprod k}
\]
is a $G \rtimes \Gamma$-cofibration.
\end{proof}

\begin{lemma}
\label{left-proper-key}
Suppose that $\calm$ is a nice monoidal model category, and that $i : X \to Y$ is a cofibration in $\calm$, regarded as a map in $\calm^{\nofc}$ concentrated at the $s$-entry for some $s \in \nofc$.  Suppose we have a diagram
\begin{equation}
\label{pushout-b}
\nicexy{
\opc \comp X \ar[d]_{i_*} \ar[r] \pushout
& A \ar[d]_-{h^A} \ar[r]^-{f}_-{\sim} \pushout
& B \ar[d]^-{h^B}
\\
\opc \comp Y \ar[r]
& A_{\infty} \ar[r]^-{f_{\infty}}
& B_{\infty}
}
\end{equation}
in $\alg(\opc) \cong \operadsigmac$ in which both squares are pushouts and $f : A \to B$ is a weak equivalence between $\Sigma$-cofibrant operads.  Then $f_{\infty}$ is also a weak equivalence between $\Sigma$-cofibrant operads.
\end{lemma}

\begin{proof}

Weak equivalences in $\alg(\opc)$ are created entry-wise in $\calm$.  The outer rectangle in \eqref{pushout-b} is also a pushout.  It follows that $h_k^A$ and $h_k^B$ are filtered in such a way that for each orbit $[r]$, the $[r]$-entry of the $k$-th map is a pushout as in \eqref{opc-jt-pushout}.   There is a commutative ladder diagram
\[
\nicexy{
A([r]) \ar@{=}[r] \ar[d]_-{f}  &
A_0([r]) \ar[d]_-{f_0} \ar[r]^-{h_1^A}
& A_1([r]) \ar[d]_-{f_1} \ar[r]^-{h_2^A}
& \cdots \ar[r]
& \colim A_k([r]) = A_\infty([r]) \ar[d]^-{f_\infty}
\\
B([r]) \ar@{=}[r]  &
B_0([r]) \ar[r]^-{h_1^B}
& B_1([r]) \ar[r]^-{h_2^B}
& \cdots \ar[r]
& \colim B_k([r]) = B_\infty([r])
}\]
in $\calm^{\Sigma_{[r]}}$. 

We now argue that all the horizontal maps $h_k^A$ and $h_k^B$ are cofibrations in $\calm^{\Sigma_{[r]}}$ and all the objects in the ladder diagram are cofibrant in $\calm^{\Sigma_{[r]}}$.
Each coproduct summand map on the left of \eqref{opc-jt-pushout} is a $\Sigma_{[r]}$-cofibration since $(-) \dotover{\Aut(T,\sds)} \sigmabrr$ is a left Quillen functor $\calm^{\Aut(T,\sds)}\to \calm^{\Sigma_{[r]}}$ and each $\id \otimes i^{\boxprod k}$ is an $\Aut(T,\sds)$-cofibration by Lemma ~\ref{equivariant-pushout-product}. But cofibrations are closed under coproducts and pushouts, so each $h_k^A$ and $h_k^B$ is a cofibration in $\calm^{\Sigma_{[r]}}$. The fact that all objects are cofibrant now follows from the $\Sigma$-cofibrancy of $A$ and $B$.

By \cite{hirschhorn} (15.10.12(1)), in order to show that the map $f_\infty$ is a weak equivalence between cofibrant objects in $\calm^{\Sigma_{[r]}}$, it suffices to show that all the vertical maps $f_k$, with $0 \leq k < \infty$, are weak equivalences by induction on $k$.

The map $f_0$ is a weak equivalence by assumption.  Suppose $k \geq 1$. Consider the commutative cube in $\calm^{\Sigma_{[r]}}$, where the coproducts are taken over the same sets of trees as in \eqref{opc-jt-pushout}:

\begin{equation*}
\label{ab-pushout-cube}
\nicexy@C-1.3cm{
\coprod
\Bigl\{[\bigotimes A(u)] \otimes Q^k_{k-1}\Bigr\} 
\dotover{\Aut(T,\sds)} \sigmabrr
\ar[dd]_-{\amalg (\Id \otimes i^{\boxprod k})_*} 
 \ar@(d,l)[dr]_-{f_*} \ar[rr] 
&&  A_{k-1}([r])
\ar[dr]^-{f_{k-1}} \ar'[d][dd] &
\\
& \coprod 
\Bigl\{[\bigotimes B(u)] \otimes Q^k_{k-1}\Bigr\} 
\dotover{\Aut(T,\sds)} \sigmabrr
\ar[dd] \ar[rr] 
&& B_{k-1}([r]) \ar[dd]
\\
\coprod 
\Bigl\{[\bigotimes A(u)] \otimes Y^{\otimes k}\Bigr\}  
\dotover{\Aut(T,\sds)} \sigmabrr
\ar@(d,l)[dr]_-{f_{*}} \ar'[r][rr] 
&& A_k([r])  \ar[dr]^-{f_{k}} &
\\
& \coprod 
\Bigl\{[\bigotimes B(u) ]\otimes Y^{\otimes k}\Bigr\} 
\dotover{\Aut(T,\sds)} \sigmabrr
\ar[rr] 
&& B_k([r])
}
\end{equation*}
Both the back face (with $A$'s) and the front face (with $B$'s) are pushout squares, and the maps from the back square to the front square are all induced by $f$.  Moreover, $f_{k-1}$ is a weak equivalence by the induction hypothesis.  By Lemma~\ref{equivariant-pushout-product}, all the objects in the diagram are cofibrant in $\calm^{\sigmabrr}$ and the vertical and diagonal maps are $\sigmabrr$-cofibrations. To see that $f_k$ in the above diagram is a weak equivalence, it is enough to show, by the Cube Lemma \cite{hovey} (5.2.6), that both maps labelled as $f_*$ are weak equivalences.

To see that the top $f_*$ in the above diagram is a weak equivalence, note that a coproduct of weak equivalences between cofibrant objects is again a weak equivalence by Ken Brown's Lemma \cite{hovey} (1.1.12).   
The left Quillen functor (definition \ref{aut-to-sigma}) $(-) \dotover{\Aut(T,\sds)} \sigmabrr$ takes $\Aut(T,\sds)$-cofibrations between $\Aut(T,\sds)$-cofibrant objects to $\sigmabrr$-cofibrations between $\sigmabrr$-cofibrant objects.
Now Ken Brown's Lemma again says that it is enough to show that within each coproduct summand, the map
\begin{equation}
\label{angq-bngq}
\nicexy{[\bigotimes A(u)] \otimes Q^k_{k-1} \ar[r]^{f_*} 
& [\bigotimes B(u)] \otimes Q^k_{k-1}}
\end{equation}
is a weak equivalence between $\Aut(t,\sds)$-cofibrant objects.  Recall that weak equivalences in any diagram category in $\calm$ are defined entrywise.  The map
\[
\nicexy{[\bigotimes A(u)] \ar[r]^-{f_*} & [\bigotimes B(u)]}
\]
is a finite tensor product of entries of $f$, each of which is a weak equivalence in $\calm$.  So this $f_*$ is a weak equivalence between cofibrant objects, and tensoring this map with $Q^k_{k-1}$ yields a weak equivalence.

A similar argument with $Y^{\otimes k}$ in place of $Q^k_{k-1}$ shows that the bottom $f_*$ in the commutative diagram is also a weak equivalence.  Therefore, as discussed above, $f_k$ is a weak equivalence, finishing the induction.

\end{proof}

\begin{theorem}
\label{oalg-model-left-proper}
If $\calm$ is a nice monoidal model category,
then the cofibrantly generated model structure on $\alg(\opc) \cong \operadsigmac$ in Corollary~\ref{oalg-model-operad} is left proper relative to the class of $\Sigma$-cofibrant operads.
\end{theorem}

\begin{proof}
The set of generating cofibrations in $\alg(\opc) \cong \operadsigmac$ is $\opc \comp \sI$, where $\sI$ is the set of generating cofibrations in $\calm^{\nofc}$, each of which is concentrated in one entry and is a generating cofibration of $\calm$ there.  A general cofibration in $\alg(\opc)$ is a retract of a relative $(\opc \comp \sI)$-cell complex. So a retract and transfinite induction argument reduces the proof to the situation in Lemma \ref{left-proper-key}.
\end{proof}

\begin{remark}\label{remark_on_bb}
An anonymous referee has pointed out that an alternative proof of Lemma \ref{left-proper-key} and Theorem \ref{oalg-model-left-proper} can be obtained using the machinery developed in \cite{bb14}.
Specifically, a modification of the proof of \cite[8.1]{bb14}, together with \cite[2.11, 2.14]{bb14} would reproduce these results.
The filtration on the pushout \eqref{pushout-b} would be different from the one we have used here, instead being based on ``classifiers.''
\end{remark}

\section{Categories of Operads Are Not Left Proper}\label{counter-example}

In this section we present an illuminating counter-example to the category of $\fC$-colored operads being left proper. The example is due to Bill Dwyer, and we thank him for allowing us to present it in this paper. 

Let $\calm$ be the category of simplicial sets with the standard (Kan) model category structure and fix $\fC=\{*\}$. In other words, we are working in just regular simplicial operads.  Let $\varnothing$ denote the initial operad, and let $\varnothing_+$ denote the operad constructed by attaching a singleton in arity $0$.  
In other words,
\begin{align*}
	\varnothing(n) &=
	\begin{cases}
		\{ \id \} & n = 1 \\
		\varnothing & n \neq 1
	\end{cases}
	&
	\varnothing_+(n) &=
	\begin{cases}
		* & n=0 \\
		\{ \id \} & n = 1 \\
		\varnothing & n > 1.
	\end{cases}
\end{align*}
The inclusion $i:\varnothing\rightarrow \varnothing_+$ is a cofibration of operads. 

Given an operad $A$, we can construct the pushout 
\[\nicexy{
\varnothing\ar[d]^{i}\ar[r] & A\ar[d]\\
\varnothing_+\ar[r] & A_+}
\]
where
$A_+(0) = \coprod_{j} A(j)/\Sigma_{j}$ and the map $A\rightarrow A_+$ is a cofibration of simplicial operads. 
If $\operadsigmac$ were left proper, then in the pushout diagram
\[\nicexy{
\varnothing \ar[d]^{i}\ar[r] & A\ar[d]\ar[r]^{f} & B\ar[d]\\
\varnothing_+\ar[r] & A_+\ar[r]^{f_+} & B_+}
\]
we would have that if $f$ is a weak equivalence, then $f_+$ is a weak equivalence. 
Taking $A$ to be an $E_\infty$-operad and $B$ to be the commutative operad, we know that $f: A \to B$ is a weak equivalence.
On the other hand, in arity 0, $f_+$ is the map
\[
	f_+(0) : \coprod_{j} A(j)/\Sigma_{j} = \coprod_{j} E\Sigma_j / \Sigma_j = \coprod_{j} B\Sigma_j \to \coprod_{j} B(j) / \Sigma_j = \coprod_j *.
\]
This is not a weak equivalence since $B\Sigma_j$ is not contractible for $j > 1$.

\end{document}